\documentclass{amsart}

\usepackage{epsfig}
\usepackage{epstopdf}
\usepackage{amscd,amsmath,amssymb}

\usepackage[english]{babel}

\newtheorem{theorem}{Theorem}[section]
\newtheorem{corollary}[theorem]{Corollary}
\newtheorem{lemma}[theorem]{Lemma}

\newtheorem{definition}[theorem]{Definition}

\numberwithin{equation}{section}

\begin{document}

\markboth{Simon King, Sergei Matveev, Vladimir Tarkaev,  Vladimir Turaev}
{Dijkgraaf-Witten $Z_2$-invariants for  Seifert manifolds}


\title{DIJKGRAAF-WITTEN $Z_2$-INVARIANTS FOR  SEIFERT MANIFOLDS}

\author{SIMON KING}

\address{Faculty of Mathematics and Natural Science\\
  Institute of Mathematics Education\\  
  Gronewaldstr. 2\\  
   D-50931 Cologne Germany \\
   simon.king@uni-koeln.de}
   
\author{SERGEI MATVEEV}

\address{Laboratory of Quantum Topology \\
  Chelyabinsk State University   \\ 
   Brat'ev Kashirinykh street\\
129,   Chelyabinsk 454001  Russia \\
 and\\
  Krasovsky Institute of Mathematics and Mechanics of  RAS\\
svmatveev@gmail.com}
  
  \author{VLADIMIR TARKAEV}

\address{Laboratory of Quantum Topology \\
  Chelyabinsk State University   \\ 
   Brat'ev Kashirinykh street\\
129,   Chelyabinsk 454001  Russia \\
 and\\
  Krasovsky Institute of Mathematics and Mechanics of  RAS\\
v.tarkaev@gmail.com}

  \author{VLADIMIR TURAEV}

\address{Department of Mathematics Indiana\\
University Bloomington IN47405, USA \\
and\\ Laboratory of Quantum Topology\\
Chelyabinsk State University, \\
   Brat'ev Kashirinykh street\\
129,   Chelyabinsk 454001  Russia\\ 
vturaev@yahoo.com}

\maketitle

\begin{abstract}
In this short paper we compute the values of  Dijkgraaf-Witten invariants over $Z_2$  for all orientable Seifert manifolds with orientable bases. 

\end{abstract}


\section{The Dijkgraaf-Witten invariants}

In 1990, Dijkgraaf and Witten \cite {DW} proposed a new
approach to constructing invariants of closed topological manifolds. Each   DW-invariant of a closed oriented  $n$-dimensional manifold $M$ is determined by a choice of a finite group $G$, a subgroup $U$ of the unitary group $ U[1]$,  and an element $h$ of the cohomology group $H^n(B; U$), where $B=B(G)$ is the classifying space of $G$. Let $S=S(M,B)$ be the set of all base point preserving   maps $M\to B$ considered up to base point preserving homotopy. The set~$S$ is finite and can be identified with the set of  homomorphisms $\pi_1(M) \to G$.
  
\begin{definition}
\emph{The Dijkgraaf-Witten invariant
of $M$ associated with~$h$} is the complex  number
\begin{center}$Z(M,h) = \frac{1}{|G|}\sum_{f\in S(M,B)}  \langle
     f^*(h),[M]\rangle,$ 
\end{center}
where $|G|$ is the order of $G$ and $[M]$ is the fundamental class of $M$.
\end{definition}
 
In this paper, we consider only the special case
where n = 3, both groups $G$ and $U$ have order 2 and
are identified with the group $\mathbb{Z}_2$. For the classifying
space  of  $\mathbb{Z}_2$ we   take the infinite-dimensional projective space
$RP^\infty$, and for $h$ we take a unique nontrivial element $\alpha^3 \in  H^3(RP^\infty;  \mathbb{Z}_2 )$, where $\alpha$ is the generator of $ H^1(RP^\infty;  \mathbb{Z}_2 )$. In this situation, the value $\langle f^*(h),[M]\rangle$ belongs to $\mathbb{Z}_2$, and the formula given above takes the form

\begin{center}
$Z(M,\alpha^3) = \frac{1}{2}\sum_{f\in S(M,B)} (-1)^{ \langle
     f^*(\alpha^3),[M]\rangle,}$ 
\end{center}

\noindent and becomes applicable to nonorientable manifolds. If
$H^1(M;  \mathbb{Z}_2 ) = 0$, then $Z(M,\alpha^3) = \frac{1}{2}$. In all other cases
$Z(M,\alpha^3$) is an integer.

\section{Quadratic function and Arf-invariant}

Let $ M $ be a  closed 3-manifold. We define a quadratic function $ Q_M
\colon H ^ 1 (M; \mathbb {Z} _2) \to \mathbb{Z}_2 $ by the  rule $ Q_M (x) =
\langle x ^ 3, [M] \rangle $, where $ x\in H ^ 1(M; \mathbb {Z} _2)$, $[M]$  is the  fundamental class of  $ M $, and $ x ^ 3\in H ^ 3 (M;
\mathbb {Z}_2) $ is the cube of  $ x $ in the sense of multiplication in
cohomology. The corresponding pairing  $\ell_M \colon H ^ 1 (M; \mathbb {Z} _2) \times H ^ 1 (M; \mathbb {Z} _2) \to \mathbb{Z}_2 $ defined by $\ell_M (x,y)= Q_M (x+y)- Q_M (x)- Q_M (y)$ is bilinear. 

The following relation between the DW-invariant of $M$ and the Arf-invariant of $Q_M$ was discovered in [MT].

\begin{theorem}\label{th1} \cite{MT}.
Let M be a closed connected 3-manifold,
and let $A \subset H^1(M;\mathbb {Z} _2) $ be the annihilator of $\ell_M$.  If there exists
 $x \in  A$  such that $x^3\neq  0$ , then $Z(M,\alpha^3) = 0$. If there
are no such elements, then $Z(M,\alpha^3) = 2^{k + m - 1}(-1)^{{\rm Arf} (Q_M)}$,
where $m$ is the dimension of $A$ and $k$ equals half the
dimension of the coset space $H^1(M; \mathbb {Z} _2)/A$.
\end{theorem}

Note that this theorem is true for orientable and   non-orientable 3-manifolds. However, it follows from the Postnikov Theorem \cite{Po} that if $M$ is orientable, then the annihilator $A$ of $\ell_M$ coincides with $H^1(M; \mathbb {Z} _2)$. Further on we will consider only orientable 3-manifolds. For brevity we will call an element $x\in H^1(M; \mathbb {Z} _2)$  {\em essential} if $x^3\neq 0$.

\begin{corollary}
Let $M$ be an orientable  closed connected 3-manifold. If there exists an essential
 $x \in  H^1(M; \mathbb {Z} _2)$, then $Z(M,\alpha^3) = 0$. If there are no such elements, then $Z(M,\alpha^3) = 2^{m - 1}$, where $m$ is the dimension of $H^1(M; \mathbb {Z} _2)$.
\end{corollary}

In view of  the above corollary, the following question is   crucial for computing DW-invariants: given a 3-manifold $M$, does   $ H^1(M; \mathbb {Z}_2)$ contain an essential element? In general, the calculation of products in cohomology and calculation of DW-invariants  is quite cumbersome, see \cite{Wa,Br,Ha}. We prefer  a very elementary method based on a nice structure of skeletons of Seifert manifolds. For simplicity, we restrict ourselves to Seifert manifolds fibered over $S^2$,  although Theorems 2 -- 4  remain  true for Seifert manifolds fibered over any closed    orientable surface. Proofs are the same.

\section{Skeletons of Seifert manifolds. }\label{4}

\begin{definition} Let $M$ be a closed 3-manifold. A 2-dimensional polyhedron  $P\subset M$ is called  {\em a skeleton of    $M$} if $M\setminus P$ consists of open 3-balls. 
\end{definition}

 Let us construct a skeleton of an orientable Seifert manifold $M=(S^2; (p_i,q_i), 1\leq i\leq n )$ fibered over $S^2$ with exceptional fibers of types $(p_i,q_i)$. Represent $S^2$ as a union of two discs $\Delta, D $ with common boundary.  Choose inside $\Delta$ disjoint discs $\delta_i, 1\leq i \leq n$, and remove their interiors. The resulting punctured disc  we denote  $\Delta_0$. Then we  join  the circles $\partial \delta_i$     with $\partial D$  by disjoint arcs    $l_i\subset  \Delta_0 , 1\leq  i \leq n$. The skeleton $P\subset M$ we are looking for is the union of the following surfaces, see Fig.\ref{Fig2}.

\begin{enumerate}
\item Annuli  $L_i = l_i\times  S^1$ and tori $t_i=\partial \delta_i \times  S^1   , 1\leq i \leq n$.

\item The torus  $T=\partial D\times S^1$, the punctured  disc $\Delta_0$,  and the disc $D$.

\item Discs  $d_i$ attached to   $t_i$ along  simple closed curves in $t_i$ of  types   $(p_i, q_i)$. 
\end{enumerate}  

Note that the surfaces of the first two types  lie in $S^2 \times S^1$ while  $d_i$ do not.  

\begin{figure}[th]
\centerline{\psfig{figure=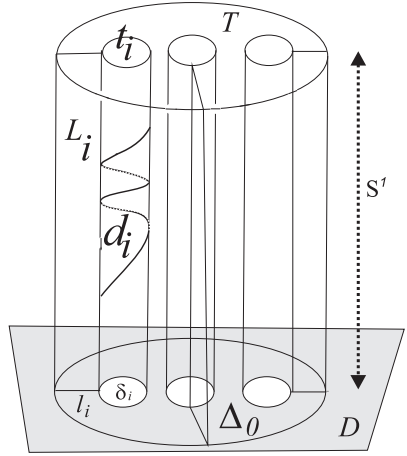,height=6cm}}
\caption{A skeleton of $M$. \label{Fig2}}
\end{figure}

\section{Seifert manifolds having $Z(M,\alpha^3)=0$} \label{S}
In this section we describe  all orientable Seifert manifolds fibered over $S^2$ whose first cohomology group contains an essential $x$.
Note that by Poincar\'e duality    $x \in H^1(M; \mathbb {Z}_2)$ is essential 
 if and only if its dual $x_\ast \in H_2(M; \mathbb {Z}_2)$           is essential in the sense that it can be realized by an odd 
 surface, i.e. by a  closed  surface having an odd Euler
 characteristic. So instead of looking for essential 1-cocycles we will construct essential 2-cycles.

Let $M=(S^2; (p_i,q_i), 1\leq i\leq n )$ be an orientable Seifert manifold fibered over $S^2$ with exceptional fibers of types 
$(p_i,q_i)$. 

\begin{theorem} 
\label{th2} Suppose that $M$ contains  an  exceptional fiber  $f_i $   of type  $(p_i,q_i)$  and an  exceptional fiber  $f_j$   of type  $(p_j,q_j)$ such that  $p_i$ is divisible by 4 while $p_j$ is even but not divisible by 4. Then there is an essential  $x\in H^1(M;\mathbb{Z}_2)$.
\end{theorem}

\begin{proof} Choose disjoint discs $D_i,D_j \subset  S^2$ containing  projection points  of $f_i,f_j$, and join their boundaries by a simple arc $c \subset S^2$. Let   $P $ be the polyhedron in $M$ consisting of the annulus $L=c\times S^1$, two tori $t_i =\partial D_i \times S^1,  t_j =\partial D_j \times S^1$, and meridional  discs   $d_i, d_j \subset M$ of  solid tori  which replace    $D_i\times S^1$  and    $D_j\times S^1$ by the standard   construction of  $M$.   The boundary curves of $d_i, d_j$ are of types $(p_i,q_i)$, $(p_j,q_j)$, respectively.
See Fig. \ref{Fig.1}, to the left.

\begin{figure}[th]
\centerline{\psfig{figure=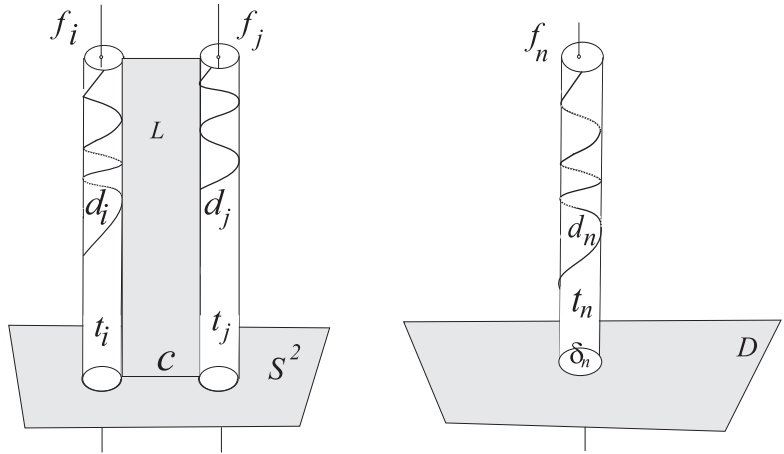,height=6cm}}
\caption{Construction of an odd surface. \label{Fig.1}}
\end{figure}

Note  that  $d_i$ can be chosen so that the circles   $\partial d_i$ and $\lambda_i=\partial L \cap t_i$  decompose    $t_i$ into  $p_i$ quadrilaterals admitting
 a black-white chessboard coloring. The same is true for $d_j$ and $t_j$. Let us remove from $P$ all $(p_i+p_j)/2 $    quadrilaterals having  the same (say, white) color. We get a closed surface  $F \subset M $. Since the Euler characteristic $\chi(P) $ of $P$ is even and $(p_i+p_j)/2 $ is odd,   $\chi(F) $ is odd. Therefore $F$ is odd and thus represents an essential 2-cycle.
\end{proof}

\begin{theorem} \label{th3} Let $M=(S^2; (p_i,q_i), 1\leq i\leq n )$ be an orientable Seifert manifold fibered over $S^2$ with exceptional fibers of types 
$(p_i,q_i)$ such that
 (1) all $p_i$ are odd and (2)  
the number of odd $q_i$ is even.  
Denote by $Q_\ast$  the sum of all $q_i$, and by  $P_\ast$  any 
alternating sum of all $p_i$ for which
 the corresponding $q_i$ are odd.   
 Suppose that the integer number    $\xi(M)=(Q_\ast+P_\ast)/2$ is odd.
    Then there exists an essential  $x\in H^1(M;\mathbb{Z}_2)$.
\end{theorem}

\begin{proof} Let us replace two exceptional fibers of types
$(p_i,q_i)$,   $(p_j,q_j)$ by  exceptional fibers of types
$(p_i,q_i+p_i)$,  $(p_j,q_j-p_j)$.   These  operations (called  
{\it parameter trading}) preserve the parity of $\xi(M)$ and produce another Seifert presentation of 
$M$, which also satisfies assumptions (1) and (2) of Theorem
\ref{th3}. Using such operations  one can easily get  all $q_i$ be
divisible  by 4,   except the last one, say, $q_n$,  which
must be  even in view of assumption (2). For this new presentation of $M$ we have $Q_\ast\equiv q_n$ {\rm mod 4} (since all other $q_i$ are divisible by 4), and $P_\ast =0$ (since there are no odd $q_i$). Taking into account that  $\xi(M)$ is odd, we may conclude that $q_n/2$ is odd.  
Let   $P$ be a skeleton of $M$  constructed for   this new presentation of $M$. Consider the union $P_0$ of the following  2-components of $P$, see Fig. \ref{Fig.1}, to the right: 
  
\begin{enumerate} 
\item The torus $t_n$ and  disc  $d_n$ attached to   $t_n$  along a simple closed curve  of type     $(p_n, q_n)$;
\item The   disc $D =S^2 \setminus {\rm Int} (\delta_n)$, where $\delta_n \subset S^2$ is the meridional disc of $t_n$.
\end{enumerate} 
  
   Then we apply the same trick as in the proof of Theorem \ref{th2} using the circle   $\partial d_n$ and taking the circle $\partial \delta_n$ instead of $\lambda_n$. 
   
   Since  $q_n$ is even,  the circles   $\partial d_n$ and $\partial\delta_n $   can be chosen so that  they     decompose    $t_n$ into  $q_n$ quadrilaterals admitting
 a black-white chessboard coloring.  Let us remove from $P_0$  all $  q_n/2 $    quadrilaterals having  the same (say, white) color. We get a closed surface  $F \subset M $. Since the Euler characteristic   of $P_0$ is even and $q_n/2 = \xi(M)  $ is odd,   $\chi(F) $ is odd. Therefore $F$ represents an essential 2-cycle.
\end{proof}

Let us introduce two classes $\mathcal{A}, \mathcal{B}$  of Seifert manifolds. Class  $\mathcal{A}$  consists of manifolds satisfying assumptions of Theorem \ref{th2},  class $\mathcal{B}$ consists of manifolds satisfying assumptions of Theorem \ref{th3}.

\begin{theorem} \label{th4}
 Let $M$ be   an orientable Seifert manifold fibered over $S^2$ with exceptional fibers of types
$(p_i,q_i), 1\leq  i \leq n$. Suppose that  $M$ contains an essential 2-cycle $x_\ast$.
Then $M$ belongs  either to $\mathcal{A}$  or to $\mathcal{B}$ 
\end{theorem}

In the proof of the theorem, we need the following self-evident lemma.
\begin{lemma} \label{44} Let $\tau$ be a torus and  $ \mathcal{C} $  a finite
collection of  simple closed curves in $\tau$ such that all their
intersection points are transverse. We will consider the union $G$
of these curves as a graph.
Suppose that the faces of $G$ have  a
black-white chess coloring  in the sense that   any edge of $G$ separates a black face from a white one.
Then any general position simple closed
curve in $\tau$ crosses the edges of $G$ at an even  number of
points.\qed
\end{lemma}

\begin{proof}[Proof of Theorem~\ref{th4}] Let $P$ be the  skeleton   of $M$  constructed    in  section \ref{4}. Denote by $\Gamma$ its
   singular graph consisting of triple and fourfold lines of $P$. The remaining part of $P$ consists
    of 2-components of $P$, that is,  of surfaces which are glued to $\Gamma$ along their
    boundary circles. Let  $B$ be the carrier of $x_\ast$, i.e. the union of  2-components of $P$
    which are   black  in the sense that they have  coefficient 1 in the 2-chain  representing $x_\ast$. All other 2-components are   white.

Case 1. Suppose that   $\Delta_0 $ and hence $D$ are white. Then the polyhedron $P_0=P\setminus \Delta_0$
is  simple   in the sense that the set of its singular points consists    of triple  lines and their crossing points. It follows that  $B $ is a closed  surface.
Since $x_\ast$ is essential, $\chi(B)$ is odd.  We may assume that $B$ is connected
 (if not then the Euler characteristic of at least one component of $B$ is odd, and we can take
it instead  of $B$). It follows that  $B$ contains   two
 annuli, say, $L_i, L_j$, discs  $d_i,d_j$, and  one of the two annuli into which the curves
 $  L_i \cap T $ and $L_j \cap T$ decompose the torus $T$. The total Euler characteristic of these 2-components is 0.

 Just as
 in the proof of Theorem \ref{th2} the circles $\partial
d_i$ and $\lambda_i= t_i\cap \partial L_i$ decompose $t_i$ into
$p_i$ quadrilateral 2-components. Since $x_\ast$ is a 2-cycle,
they are colored according to the chessboard rule. It follows that
$p_i$ and    $p_j$  are even. The remaining part of $B$
consists of $(p_i+p_j)/2$ black quadrilaterals in $t_i$ and $t_j$.
Since $\chi(B)$ is odd, we may conclude that one of $p_i/2, p_j/2
$ is even while the other is odd. Therefore $M\in \mathcal{A}$.

Case 2. Suppose that $\Delta_0$ is black. Then each  $d_i$ is  black. 
 Otherwise the boundary of $x_\ast $ would contain $\partial
\delta_i$, which is impossible since $x_\ast $ is a 2-cycle.

Let us prove that if $L_i$ is black then $q_i$ is odd.   Consider the graph $G=\partial d_i\cup \lambda_i \subset t_i$. As above, $x_\ast$ induces a  chess black-white  coloring of its faces. Then a parallel copy $\partial \delta'_i$ of $\partial \delta_i$ crosses $\lambda_i$ at one point and crosses $\partial d_i$ at $q_i$  points. By Lemma \ref{44} the number $q_i+1$ must be even, which means that $q_i$ is odd. Similarly,  if $L_i$ is white then $q_i$ is even.

 Let us prove that all $p_i$ are odd. Suppose that  $L_i$ is black. Consider the graph $G=\partial d_i\cup \partial \delta_i  \subset t_i$.   It decomposes $t_i$ into black-white colored faces. The coloring is induced by the 2-cycle $x_\ast$.  Let $\lambda'_i \subset t_i$ be a parallel copy    of $\lambda_i$. Then $\lambda'_i$  crosses $\partial d_i$ at $p_i$ points and crosses     $\partial \delta_i $ at  one point. By Lemma 1 the number $p_i+1$ must be even, which means that $p_i$ is odd. Suppose that  $L_i$ is white then $q_i$ is even. Therefore ,  $p_i$, being   coprime with $Q_i$, is odd.

Let us prove that the number of  black $L_i$  and hence the number of    odd $q_i$ are even.  This is because the boundary circles of black $L_i$ decompose $T$ into  black-white  colored  annuli such any two neighboring annuli have different colors.  Therefore  any meridional circle of $T$, for example, $\partial D$,  crosses  those boundary circles at an even number of points. 

Let us prove that $\xi(M)$ is odd. Just as in the proof of Theorem \ref{th3} we transform the  given Seifert presentation  of  $M$ into a new Seifert presentation such that all $q_i, 1\leq i\leq n-1,$ are divisible  by 4 and  $q_n$ is even. Then the  carrier $B$ of $x_\ast$  consists of the following black surfaces:
\begin{enumerate} 
\item      $ \Delta_0$ and  $D$;
 \item   All  discs $d_i, 1\leq i \leq  n $; 
 \item black quadrilaterals  contained in  $t_i, 1\leq i\leq n$. Each $t_i$ contains $q_i/2$  black quadrilaterals, where $(p_i,q_i)$  is the type of the corresponding $d_i$. Of course  $B$  may contain the torus  $T$, but it is      is homologically trivial and thus can be neglected.
\end{enumerate}
 
 Note that  $B$ has only triple singularities and thus is a closed surface. Since $x_\ast$ is   essential,  $\chi(B)$  is odd.  Now  we calculate  $\chi(B)$ by counting Euler characteristics of the above black surfaces.  
We get $\chi(B)=q_n/2$ {\rm mod} 2. It follows that  $q_n/2$ is odd. Since all $q_i$ are even,   $P_\ast =0$.   Taking into account that  all $q_i, 1\leq i\leq n-1$ are divisible by 4 and that $q_n/2$ is odd we may conclude that  $Q_\ast =(\sum_{i=1}^{n}q_i)/2 $ is  odd. Therefore, $\xi(M)$ is odd. It follows that  is in class $\mathcal{B}$.
\end{proof}

\section*{Acknowledgments}
S. Matveev, V. Tarkaev and V. Turaev were supported in part by the Laboratory of Quantum Topology, Chelyabinsk State University (contract no. 14.Z50.31.0020). S. Matveev and V. Tarkaev were supported in part by the Ministry of Education and Science of the Russia (the state task number 1.1260.2014/K) and the Russian Foundation for Basic Research (project no. 14-01-00441).

\end{document}